\documentclass[12pt,reqno]{amsart}

\numberwithin{equation}{section}

\newtheorem{teo}{Theorem }[section]

\newtheorem{lem}[teo]{Lemma}
\newtheorem{prop}[teo]{Proposition}

\newtheorem{rem}{Remark}

\begin{document}


\title[KAWAHARA EQUATION ON A HALF-STRIP]
      { THE 2D KAWAHARA EQUATION ON A HALF-STRIP }
\author[ N.~A. Larkin]{Nikolai A. Larkin$^{\dag}$}

\address
{
Departamento de Matem\'atica, Universidade Estadual
de Maring\'a, Av. Colombo 5790: Ag\^encia UEM, 87020-900, Maring\'a, PR, Brazil
}

\email{$^{\dag}$nlarkine@uem.br}

\thanks{}
\thanks{ N. A. Larkin was supported by Funda\c{c}\~ao
Arauc\'aria, Estado do Paran\'a, Brasil} \subjclass[]{}


\begin{abstract}
We formulate on a half-strip an initial boundary value problem for
the two-dimensional Kawahara equation.
 Existence and uniqueness of a regular solution as well as the exponential decay rate for the elevated norm
$$\|u\|^2_{H^1(D)}(t)+\|u\|^2_{L^2(D)}(t)$$
  of small solutions   as $t \rightarrow \infty$ are proven.
\end{abstract}

\maketitle

\section{\bf Introduction}

We are concerned with an initial boundary value problem (IBVP) posed
on a half-strip for the 2D Kawahara equation (KZK)
$$ u_t+(\alpha +u)u_x +u_{xxx}+u_{xyy} -\partial^5_x u=0    \eqno(1.1)$$
 which is a two-dimensional analog of the well-known Kawahara
equation, \cite{phys,fam3,kawa},
$$ u_t+(\alpha+u)u_x +u_{xxx}-\partial^5_x u=0,  \eqno(1.2)$$
where $\alpha $ is  equal  to 1 or to 0.
 The theory of the Cauchy problem for (1.2) and other dispersive
equations like the KdV equation has been extensively studied and is
considerably advanced today
\cite{biagioni,bona2,bourgain,cui,kato,ponce2,saut2,temam1}. In
recent years, results on IBVPs  for dispersive equations both in
bounded and unbounded domains have appeared
\cite{bona3,bubnov,chile,colin,fam2,familark,larkin,lar2}. It was
discovered in \cite{larkin,marcio} that the KdV and Kawahara
equations have an implicit internal dissipation. This allowed the
proof of exponential decay of small solutions in bounded domains
without adding any artificial damping term. Later, this effect was
proven for a wide class of dispersive equations of any odd order
with one space variable \cite{familark}.
\par On the other hand, it has been shown in \cite{rosier2}
that control of the linear KdV equation with the linear transport
term $u_x$ (the case $\alpha=1$) may fail for critical domains. It
means that there is no decay of solutions  for a set of critical
domains, hence, there is no decay of solutions in a quarter-plane
without inclusion into equation of some additional internal damping.
More recent results on control and stabilization for the KdV
equation can be found in \cite{rozan}. Nevertheless,  it is possible
to prove the exponential decay rate of small solutions for the KdV
and Kawahara equations posed on any bounded interval neglecting the
transport term (the case $\alpha=0$) \cite {larluc,marcio}.
\par As far as the ZK equation is concerned, there are some recent results
 on the Cauchy problem and IBVP
\cite{faminski,fam2,fara,pastor,pastor2,ribaud,saut2,saut3,saut4}.
 Our work was motivated by \cite{saut3,saut4} on IBVP for the ZK equation posed on bounded domains and on a  strip
unbounded in $y$ variable.
 Studying this paper, we have found that in the case of the ZK equation posed on a half-strip
(which simulates a flow in a channel)
 the walls of the channel and the term
$u_{xyy}$  deliver additional "dissipation" which helped to prove
decay of small solutions in domains of a channel type unbounded in
$x$ direction \cite{ h-strip,h1}.
\par Publications on dispersive multidimensional equations of a higher
order (such as the KZK equation) appeared quite recently and were
concerned with the existence of weak solutions, \cite{fam3}, and
physical motivation \cite{phys}.
\par We study (1.1) on a half-strip
 $$
 D=\left\{(x,y) \in \mathbb{R}^{2}: \quad x>0, \quad y \in (0,L)
\right\}
$$
and establish   exponential decay of small solutions   even for
$\alpha=1$ provided that $L$ is not too large. If $\alpha =0$, we
obtain the exponential decay rate of small solutions for any finite
$L$. We limit our scope, from technical reasons, to homogeneous
boundary conditions, but it is also possible to consider
nonhomogeneous ones. More precisely, we formulate in Section 2 the
IBVP (2.1)-(2.4). In order to demonstrate existence of global
regular solutions, we exploit the Faedo-Galerkin method. Estimates,
independent of the parameter of approximations $N$, permit us to
establish the existence of regular solutions for the original
problem (2.1)-(2.4). We prove these estimates in Section 3. \\
Surprisingly, we did not succeed to prove global existence for all
positive weights $e^{kx}$ as in \cite{h1,h-strip} and  imposed a
restriction $3-5k^2>0.$ Our condition for the width of a channel,
$0<L<\pi$, is more precise then $0<L<2\sqrt{2}$ in \cite{h1,h-strip}
due to the sharp estimate
$$\|u\|^2_{L^2(D)}(t)\leq \frac{L^2}{\pi^2}\|u_y\|^2(t)_{L^2(D)}$$
instead of
$$\|u\|^2_{L^2(D)}(t)\leq \frac{L^2}{8}\|u_y\|^2(t)_{L^2(D)}$$
used in \cite{h1,h-strip}.

In Section 4, we pass to the limit as $N \to \infty$ and obtain a
global regular solution of (2.1)-(2.4). In Section 5, we prove
uniqueness of a regular solution. Finally, in Section 6, we
establish the exponential decay rate for the elevated norm \quad
$\|u\|^2_{H^1(D)}(t)+\|u_{xx}\|^2_{L^2(D)}(t)$\quad of small
solutions both for $\alpha =1$ and for $\alpha =0.$

\section{Formulation of the problem }

Let $T,L$ be real positive numbers;
\begin{eqnarray} && D= \left\{(x,y) \in \mathbb{R}^{2}: \quad x>0, \quad y \in (0,L) \right\}; \nonumber \\
&& Q_t= D \times (0,t), \quad t \in (0,T). \nonumber
\end{eqnarray}
Consider in $Q_t$ the following IBVP:
\begin{eqnarray}
&& Lu \equiv u_t + \alpha u_x+ uu_x +\Delta u_x-\partial^5_x u=0 \quad \textrm{in} \quad Q_t; \label{1.1} \\
&& u(0,y,t)= u_x(0,y,t)=u(x,0,t)=u(x,L,t)=0, \nonumber \\
&& \qquad \qquad \qquad y \in (0,L), \quad x>0, \quad t>0; \label{1.2}\\
&&
\lim_{x\to\infty}u(x,y,t)=\lim_{x\to\infty}u_x(x,y,t)=\lim_{x\to\infty}u_{xx}(x,y,t),\\
&& u(x,y,0)=u_0(x,y), \quad (x,y) \in D. \label{1.3}
\end{eqnarray}
Here $\partial_x^j={\partial^j}/{\partial x^j}$,
$\partial_y^j={\partial^j}/{\partial y^j}$,
 ${\Delta}= \partial_x^2 + \partial_y^2$,\;$\alpha = 0$ or $1$. We adopt the usual
 notations $H^k$ for $L^2$-based Sobolev spaces;\,
 $\| \cdot \|$ and $(\cdot , \cdot )$ denote the norm and the scalar product in $L^2(D)$, ${|\nabla u |}^2=u_x^2+u_y^2$.

\section{Existence Theorem}

\begin{teo}\label{T1} Let $T$, $L$ be arbitrary real positive numbers, $\alpha = 1$ and $k$ be a real positive number
such that $3-5k^2>0.$. Given $u_0(x,y)$  such that
\begin{eqnarray}
&& u_0 \in H^2(D), \quad  (\Delta u_{0x}+ \partial_x^5 u_0)\;\in L^2(D), \nonumber \\
&& u_0(0,y)=u_{0x}(0,y)=u_0(x,0)=u_0(x,L)=0, \nonumber
\end{eqnarray}
\begin{eqnarray}
&&J_w\equiv \displaystyle\int_D e^{kx} \{u_0^2 + {|\nabla
u_0|}^2+ {|\partial^2_y u_0|}^2+ {|\partial^2_x u_0|}^2\nonumber\\
&&+[\partial_x^5u_0+ \Delta u_{0x}]^2 \} dxdy < \infty ,\nonumber
\end{eqnarray}
 there exists a unique regular solution of
(\ref{1.1})-(\ref{1.3}):
\begin{eqnarray*}
&& u \in L^{\infty} (0,T;H^2(D)) \cap L^2 (0,T;H^3(D)),  \\
&& \partial_x^4 u,\; \partial_x^5 u \;\in L^2(0,T;L^2(D)),\;u_{xxy}\in L^{\infty}(0,T;L^2(D));\\
&&(\partial_x^5 u+ \Delta u_x )\in L^{\infty} (0,T;L^2(D)) \cap L^2(0,T;H^1(D)), \\
&& u_t \in L^{\infty} (0,T;L^2(D)) \cap L^2(0,T;H^1(D)). \\
\end{eqnarray*}
\end{teo}
\begin{rem} Obviously, for all $k$ satisfying the conditions of
Theorem 3.1, there is a real positive number $a$ such that
\begin{equation}
3-5k^2=2a. \end{equation}
\end{rem}
\begin{proof}

{\bf Approximate Solutions.}

To prove the existence part of this theorem, we put $\alpha =0$ and
use the Faedo-Galerkin Method as follows:\\ for all $N$ natural, we
define an approximate solution of (\ref{1.1})-(\ref{1.3}) in the
form
\begin{equation}
u^N(x,y,t)=\sum^N_{j=1}\omega_j(y)g_j(x,t), \label{2.1}
\end{equation}
where $\omega_j(y)$ are orthonormal in $L^2(0,L)$ eigenfunctions of
the following Dirichlet problem:
\begin{eqnarray}
&&-\omega_{jyy}(y)=\lambda_j \omega_j(y),\;y\in(0,L);\nonumber\\
&&\omega_j(0))=\omega_j(L)\nonumber
\end{eqnarray}
and $g_j(x,t)$ are solutions to the following initial boundary value
problem for the system of N generalized KdV equations:
\begin{eqnarray}
&&\frac{\partial}{\partial t}g_j(x,t)+\sum_{l,k=1}^N
a_{lkj}g_l(x,t)g_{kx}(x,t)+\partial^3_x g_j(x,t)\nonumber \\
&& -\partial^5_x g_j(x,t)-\lambda_j g_{jx}(x,t)=0, \nonumber  \\
&& g_j(0,t)=g_{jx}(0,t)=0, \;g_j(x,0)=u_{0j}(x)\label{eN},
\end{eqnarray}
where
\begin{eqnarray}
&&a_{klj}=\int_0^L \omega_k(y)\omega_l(y)\omega_j(y)\,dy,
\;j,k,l=1,...,N;\nonumber\\
&&u_{0j}(x)=\int_0^L u_0(x,y)\omega_j(y) dy.\nonumber
\end{eqnarray}
 Solvability of (3.3) (at least local in t) follows from
\cite{kuvsh,lar-dor,lar2}. Hence, our goal is to prove necessary a
priori estimates, uniform in  $N$, which will permit us to pass to
the limit in (\ref{eN}) as $N\to \infty$ and to establish the
existence result. We assume first that a function $u_0$ is
sufficiently smooth to ensure  calculations.  Exact conditions for
$u_0$ will follow from a priori estimates for $u^N$ independent of
$N$ and usual compactness arguments.
 \begin{rem} We put $\alpha =0$ for technical reasons. The case
$\alpha=1$ does not change the proof of Theorem 3.1.
\end{rem}

{\bf ESTIMATE I.} Multiplying the $j$-equation of (\ref{eN}) by
$g_j(x,t)$, summing over $j=1,..,N$ and integrating the result with
respect to $x$ over $R^+$, we obtain
\begin{eqnarray}
&&\frac{1}{2}\frac{d}{dt}\|u^N\|^2(t)+(|u^N|^2,
u_x^N)(t)+(u^N,\partial^3_x
u^N)(t)\nonumber\\
&&-(u^N,\partial^5_x u^N)(t)+(u^N,\partial^2_y u_x^N)(t)=0.\nonumber
\end{eqnarray}
In our calculations we will drop the index $N$ where this is not
ambiguous. Integrating by parts the last equality, we get

$$\frac{d}{dt}\|u\|^2(t)+\int^L_0 u_{xx}^2(0,y,t)\,dy=0.$$

It follows from here that for $N$ sufficiently large and $\forall
t>0$

\begin{equation}
\|u^N\|^2(t)+\int^t_0\int^L_0 |u^N_{xx}(0,y,\tau)|^2 \,dyd\tau
=\|u^N\|^2(0)\leq 2\|u_0\|^2. \label{E1}
\end{equation}

{\bf ESTIMATE II.} Multiplying the $j$-equation of (\ref{eN}) by
$e^{kx}g_j(x,t)$, summing over $j=1,..,N$ and integrating the result
with respect to $x$ over $R^+$, we obtain
\begin{equation}
(u^N_t+u^N u_x^N+\partial^3_x u^N-\partial^5_x u^N+\partial^2_y
u_x^N, e^{kx}u^N)(t)=0. \nonumber
\end{equation}

Integrating by parts and dropping the index $N$, we deduce
\begin{eqnarray}
&&\frac{d}{dt}(e^{kx},u^2)(t)+(3k-5k^3)(e^{kx},u^2_x)(t)+5k(e^{kx},u_{xx}^2)(t)\nonumber
\\
&& \int_0^L u_{xx}^2(0,y,t)(t)\,dy + k(e^{kx},u^2_y)(t)
+(k^5-k^3)(e^{kx},u^2)(t)\nonumber\\
&&=\frac{2k}{3}(e^{kx},u^3)(t). \label{e21}
\end{eqnarray}

In our calculations, we will frequently use the following
multiplicative inequalities \cite{lady2}:
\begin{prop}
i) For all $u \in H^1(\mathbb{R}^2)$
\begin{equation} {\|u\|}_{L^4(\mathbb{R}^2)}^2 \leq 2 {\|u\|}_{L^2(\mathbb{R}^2)}{\|\nabla u\|}_{L^2(\mathbb{R}^2)}.
 \label{p1}
\end{equation}
\qquad \qquad \qquad \qquad ii) For all $u \in H^1(D)$
\begin{equation} {\|u\|}_{L^4(D)}^2 \leq C_D {\|u\|}_{L^2(D)}{\|u\|}_{H^1(D)}, \label{p2}
\end{equation}
where the constant $C_D$ depends on a way of continuation of $u \in
 H^1(D)$ as $ \tilde{u}(\mathbb{R}^2)$ such that $\tilde{u}(D)=u(D).$
\end{prop}

Extending $u$ by zero into the exterior of $D$ and making use of
(\ref{E1}), we estimate
\begin{eqnarray*}
&&I= \frac{2k}{3}(e^{kx},u^3)(t) \leq
\dfrac{4k}{3}{\|e^{\frac{kx}{2}}u\|}(t){\|
\nabla(e^{\frac{kx}{2}}u)\|}(t){\|u \|}(t) \\
&&\qquad   \leq C(k,\epsilon)\displaystyle\sup_{{\mathrm{t \in
(0,T)}}}{\|u \|}^2(t){\|
e^{\frac{kx}{2}}u\|}^2(t) + \dfrac{\epsilon k}{4}{\|\nabla(e^{\frac{kx}{2}}u)\|}^2(t) \\
&&\qquad   \leq C(k,\epsilon)\|u_0 \|^2{\|
 e^{\frac{kx}{2}}u\|}^2(t) + \frac{\epsilon k^3}{8}(e^{kx},u^2)(t)  \\
&&\qquad  + \dfrac{\epsilon k}{2}(e^{kx},u_x^2)(t) +\dfrac{\epsilon
k}{4}(e^{kx},u_y^2)(t).
\end{eqnarray*}

Differently from the case of the Zakharov-Kuznetsov equation, see
\cite{h1}, we do not have Estimate II for all positive $k$ because
of the term $(3k-5k^2)(e^{kx},u^2_x)(t)$ in (3.5) which has to be
positively defined. This implies $k(3-5k^2)>0$. Henceforth, we will
put $3-5k^2=2a>0$, where $a$ is a real positive number. Taking this
into account, we substitute $I$ into (\ref{e21}) and obtain  for
$\epsilon>0$ sufficiently small the following inequality:
 \begin{eqnarray}
&&\frac{d}{dt}(e^{kx},u^2)(t)+(e^{kx},u^2_x)(t)+(e^{kx},u_{xx}^2)(t)\nonumber
\\
&& +\int_0^L u_{xx}^2(0,y,\tau)(\tau)\,dy + (e^{kx},u^2_y)(t)
\nonumber\\
&&\leq C(k,\|u_0\|)(e^{kx},u^2)(t). \label{e22}
\end{eqnarray}

By the Gronwall lemma,
$$ (e^{kx},u^2)(t)\leq C(T,k, \|u_0\|)(e^{kx}, u_0^2).$$
Returning to (\ref{e22}) gives
 \begin{eqnarray}
&&(e^{kx},|u^N|^2)(t)+\int_0^t(e^{kx},|u^N_{xx}|^2+|\nabla u^N|^2
)(\tau)\,d\tau \nonumber\\
&&+\int_0^t\int_0^L |u_{xx}^N(0,y,\tau)|^2\,dyd\tau)\leq C(T,k,
\|u_0\|)(e^{kx}, u_0^2),\label{E2}
\end{eqnarray}

where the constant $C$ does not depend on $N$.

{\bf ESTIMATE III.} Taking into account the structure of
$u^N(x,y,t)$, consider the scalar product
\begin{equation}
-2(e^{kx}\partial^2_y u^N,[u^N_t+u^N u_x^N+\partial^3_x
u^N-\partial^5_x u^N+\partial^2_y u^N_x])(t)=0. \nonumber
\end{equation}
Acting as by proving Estimate II and dropping the index $N$, we come
to the following equality:
\begin{eqnarray}
&&\frac{d}{dt}(e^{kx},u^2_y)(t)+2ak(e^{kx},u^2_{xy})(t)+5k(e^{kx},u^2_{xxy}(t))\nonumber\\
&&+\int^L_0 u_{xxy}^2(0,y,t)\,dy
+(k^5-k^3)(e^{kx},u^2_y)(t)+k(e^{kx},u^2_{yy}(t)\nonumber \\
&&+2(u_y,e^{kx} [u_yu_x+u_{xy}u])(t)=0. \label{e31}
\end{eqnarray}

We estimate
$$ I=(u_y,e^{kx} [u_yu_x+u_{xy}u])(t)=\underbrace{(u_x ,e^{kx}u_y^2)(t)}
_{I_1}+ \underbrace{(u, e^{kx}u_yu_{xy})(t)}_{I_2}.$$

Since $u_y \big |_{y=0,L} \neq 0$, we cannot extend $u(x,y,t)$ by
zero into the exterior of $D$ and cannot use inequality (\ref{p1}).
Instead, we use (\ref{p2}): $${\|u\|}_{L^4(D)}^2 \leq C_D
{\|u\|}_{L^2(D)}{\| u\|}_{H^1(D)},$$ where the constant $C_D$ does
not depend on a measure of $D$.
\begin{eqnarray*}
&&I_1=(u_xe^{kx}u_y^2)(t) \leq {\|u_x \|}(t){\|e^{\frac{kx}{2}}u_y \|}_{L^4(D)}^2(t)\\
&& \leq C_D{\|u_x \|}(t){\|e^{\frac{kx}{2}}u_y \|}(t){\|e^{\frac{kx}{2}}u_y \|}_{H^1(D)}(t)\\
&&\qquad  \leq C(\delta){\|u_x \|}^2(t){\|e^{\frac{kx}{2}}u_y \|}^2(t)+ \delta {\|e^{\frac{kx}{2}}u_y \|}_{H^1(D)}^2(t) \\
&&\qquad \leq C(\delta){\|u_x \|}^2(t){\|e^{\frac{kx}{2}}u_y \|}^2(t)
+ \delta(1+\frac{k^2}{2}){\|e^{\frac{kx}{2}}u_y \|}^2(t) \\
&&\qquad +2\delta{\|e^{\frac{kx}{2}}u_{yx} \|}^2(t)
+\delta{\|e^{\frac{kx}{2}}u_{yy} \|}^2(t),
\end{eqnarray*}
\begin{eqnarray*}
 &&I_2=(u, e^{kx}u_yu_{xy})(t)= \dfrac{1}{2}(u, e^{kx}{(u_y^2)}_x)(t) \\
&&\qquad = -\dfrac{k}{2} (u, e^{kx}u_y^2)(t) - \dfrac{1}{2}( e^{kx}u_x,u_y^2)(t) \\
&&\qquad \leq C(k,\delta){\|u \|}^2(t){\|e^{\frac{kx}{2}}u_y
\|}^2(t)+C(\delta){\|u_x \|}^2(t){\|e^{\frac{kx}{2}}u_y \|}^2(t)  \\
&&\qquad +4\delta{\|e^{\frac{kx}{2}}u_{yx} \|}^2(t)+2\delta
{\|e^{\frac{kx}{2}}u_{yy} \|}^2(t)+2\delta \left(1+ \frac{k^2}{2}
\right)\|e^{\frac{kx}{2}}u_y \|^2(t),
\end{eqnarray*}
where $\delta$ is an arbitrary positive constant. \\
Substituting $I_1-I_2$ into (\ref{e31}), taking $\delta
>0$ sufficiently small and using (\ref{E1}),\;(\ref{E2}), we come to the
inequality
\begin{eqnarray}
&& \dfrac{d}{dt}{\|e^{\frac{kx}{2}}u_y
\|}^2(t)+\displaystyle\int_{0}^{L}u^2_{yxx}(0,y,t)\,dy \nonumber\\
&&+(e^{kx},[u_{xy}^2+u_{yy}^2+u^2_{xxy}])(t) \leq C(k, \delta)[ {\|e^{\frac{kx}{2}}u_y \|}^2(t) \nonumber\\
&&  +[ {\|u \|}^2(t)+{\|u_x \|}^2(t)]{\|e^{\frac{kx}{2}}u_y
\|}^2(t)]. \label{e32} \end{eqnarray} Making use the Gronwall lemma
and Estimates I, II, we find
\begin{eqnarray*}
{\|e^{\frac{kx}{2}}u_y \|}^2(t)& \leq & (e^{kx}, u_{0y}^2)e^{C(k,
 \delta)\displaystyle\int_{0}^{t}\left[{\|u \|}^2(\tau)+{\|u_x \|}^2(\tau)+1 \right]d\tau.}  \\
& \leq & (e^{kx}, u_{0y}^2)e^{C(k, \delta, \|u_0\|,T)(e^{kx},u_0^2)}
\leq C(e^{kx}, u_{0y}^2).
\end{eqnarray*}
Integrating (\ref{e32}) over $(0,t)$ gives
\begin{eqnarray}
& (e^{kx},|u^N_y|^2)(t)
+ \displaystyle\int_{0}^{t}(e^{kx},[|u^N_{xy}|^2+|u^N_{yy}|^2+|u^N_{xxy}|^2])(\tau)\, d\tau \nonumber \\
&+\displaystyle\int_{0}^{t}\int^L_0 |u^N_{xxy}(0,y,\tau)|^2\,dy
d\tau  \leq C(k,T,\|u_0\|)(e^{kx},u_{0y}^2). \label{E3}
\end{eqnarray}

{\bf ESTIMATE IV.}  Dropping the index $N$, transform the scalar
product
$$2(e^{kx}\partial^4_y u^N,[u^N_t+u^N u^N_x+\partial^3_x
u^N-\partial^5_x u^N+\partial^2_y u^N_x])(t)=0
$$
into the following equality:
\begin{eqnarray}
&&\dfrac{d}{dt}(e^{kx},u_{yy}^2)(t)
+2ak(e^{kx},({|D_y^2u_x|}^2)(t)+k(e^{kx},{|D_y^3u|}^2)(t) \nonumber \\
&& +5k(e^{kx},({|D_y^2u_{xx}|}^2)(t)+\displaystyle\int^L_0
|\partial^2_y u_{xx}(0,y,t)|^2\,dy
+(k^5-k^3)(e^{kx},({|D_y^2u|}^2)(t)\nonumber\\
&& +2(u_{yy}e^{kx},{(uu_x)}_{yy})(t)=0. \label{e4}
\end{eqnarray}
Denote
$$I=(u_{yy}e^{kx},{(uu_x)}_{yy})(t)=\underbrace{(e^{kx}u_{yy}^2,u_x)(t)}_{I_{1}}
+2\underbrace{(e^{kx}u_{yy},u_yu_{xy})(t)}_{I_{2}} +
\underbrace{(e^{kx}u_{yy},u_{xyy}u)(t)}_{I_{3}}.$$ Making use of
(\ref{p1}) and (\ref{p2}), we estimate for all $\delta>0$
\begin{eqnarray*}
&&I_{1}=(e^{kx}u_{yy}^2,u_x)(t) \leq {\|u_x
\|}(t)\|e^{\frac{kx}{2}}u_{yy} \|_{L^4(D)}^2(t)\\
&&\leq 2\|u_x \|(t){\|e^{\frac{kx}{2}}u_{yy} \|}(t){\|\nabla(e^{\frac{kx}{2}}u_{yy} )\|}(t) \\
&& \leq C(\delta){\|u_x \|}^2(t){\|e^{\frac{kx}{2}}u_{yy} \|}^2(t) + \delta\frac{k^2}{2}{\|e^{\frac{kx}{2}}u_{yy} \|}_{L^2(D)}^2(t) \\
&& +2\delta{\|e^{\frac{kx}{2}}u_{yyx} \|}^2(t)+\delta {\|e^{\frac{kx}{2}}u_{yyy} \|}^2(t), \\
&&I_{2}=2(e^{kx}u_{yy},u_yu_{yx})(t)=(e^{kx}u_{yy},{(u_y^2)}_x)(t)\nonumber  \\
&& =-\underbrace{k(e^{kx}u_{yy},u_y^2)(t)}_{I_{21}}
 - \underbrace{(e^{kx}u_{yyx},u_y^2)(t)}_{I_{22}}; \\
&&I_{21} \leq k{\| e^{kx}u_y^2 \|}(t){\| u_{yy} \|}(t)  \leq C(k){\|
e^{\frac{kx}{2}}u_y \|}(t){\| e^{\frac{kx}{2}}u_y \|}_{H^1(D)}(t){\|
u_{yy} \|}(t) \\
&&\leq \delta{\| e^{\frac{kx}{2}}u_y \|}^2(t) {\|
e^{\frac{kx}{2}}u_y \|}_{H^1(D)}^2(t)
+C(\delta,k){\| e^{kx}u_{yy} \|}^2(t),\\
&&I_{22} \leq \delta {\| u_{yyx} \|}^2(t) + C(\delta){\|
e^{\frac{kx}{2}}u_y \|}^2(t){\| e^{\frac{kx}{2}}u_y
\|}_{H^1(D)}^2(t),\\
&&I_{3}= (e^{kx}u_{yy},u_{yyx}u)(t) = \frac{1}{2}(e^{kx}u,{(u_{yy}^2)}_x)(t) \\
&&= -\underbrace{\dfrac{k}{2}
(e^{kx}u,u_{yy}^2)(t)y}_{I_{31}}-\underbrace{ \frac{1}{2}
 (e^{kx}u_x,u_{yy}^2)(t)}_{I_{32}}, \\
&&I_{31}\leq \dfrac{k}{2}{\|  u\|}(t){\|e^{\frac{kx}{2}}u_{yy}\|}_{L^4(D)}^2(t) \\
&&  \leq C(\delta,k){\| u\|}^2(t){\|e^{\frac{kx}{2}}u_{yy}\|}^2(t)
+\delta\frac{k^2}{2}{\|e^{\frac{kx}{2}}u_{yy}  \|}^2(t) \\
&& +2\delta{\|e^{\frac{kx}{2}}u_{yyx}  \|}^2(t)+ \delta {\|e^{\frac{kx}{2}}u_{yyy}  \|}^2(t), \\
&&I_{32} \leq C(\delta,k){\| u_x \|}^2(t){\|e^{\frac{kx}{2}}u_{yy}\|}^2(t)
+\delta\frac{k^2}{2}{\|e^{\frac{kx}{2}}u_{yy}\|}^2(t) \\
&& +2\delta{\|e^{\frac{kx}{2}}u_{yyx}  \|}^2(t)+ \delta
{\|e^{\frac{kx}{2}}u_{yyy}  \|}^2(t).
\end{eqnarray*}
Taking  $\delta >0$ sufficiently small and substituting $I_1-I_3$
into (\ref{e4}), we obtain
\begin{eqnarray}
&&\dfrac{d}{dt}(e^{kx},u_{yy}^2)(t)
+(e^{kx},[{|D_y^2u_x|}^2+{|D_y^3u|}^2+{|D_y^2u_{xx}|}^2])(t) \nonumber \\
&&+\displaystyle\int_{0}^{L}{|D_y^2u_{xx}(0,y,t)|}^2dy  \leq C(k){\| e^{\frac{kx}{2}}u_y \|}^2(t){\| e^{\frac{kx}{2}}u_y \|}_{H^1(D)}^2(t) \nonumber \\
&& +C(k)\left[ {\| u_x \|}^2(t)+{\| u
\|}^2(t)+1\right](e^{kx},u_{yy}^2)(t).
\end{eqnarray}
The previous estimates and the Gronwall lemma yield
\begin{eqnarray}
&& ( e^{kx},|u^N_{yy}|^2)(t)
+\displaystyle\int_{0}^{t}( e^{kx},[{|D_y^2u^N_x|}^2+{|D_y^3u^N|}^2+{|D_y^2u^N_{xx}|}^2])(\tau)\,d\tau \nonumber \\
&&+\displaystyle\int_{0}^{t}\displaystyle\int_{0}^{L}{|D_y^2u^N_{xx}(0,y,\tau)|}^2dyd\tau\leq
C(e^{kx},u_{0y}^2+u_{0yy}^2), \label{E4}
\end{eqnarray}
where the constant $C$ does not depend on $N.$

\begin{prop} \label{psup}

Let $u\in H^2(D)$ such that $u_{xxy},\,u_{yyx}\in L^2(D)$ and
$u(x,0,t)=u(x,L,t)=u(0,y,t)=0.$ Then
\begin{eqnarray}
&&\displaystyle\sup_D u^2(x,y,t)\leq
\|u\|^2(t)+\|\nabla u\|^2(t)+\|u_{xy}\|^2(t),\nonumber\\
&&\displaystyle\sup_D u^2_x(x,y,t)\leq \|u_x\|^2(t)+\|\nabla
u_x\|^2(t)+\|u_{xxy}\|^2(t), \nonumber\\
&&\displaystyle\sup_D u^2_y(x,y,t)\leq \|u_y\|^2(t)+\|\nabla
u_y\|^2(t)+\|u_{xyy}\|^2(t),\nonumber.
\end{eqnarray}
\end{prop}

\begin{proof} We will prove the last inequality; the others can be
proven in the same manner. Due to boundary conditions
$u(x,0,t)=u(x,L,t)=0$, there is a point $y=m, \;m\in (0,L)$ for
fixed $(x,t)$ such that $u_y(x,m,t)=0.$ It implies
\begin{eqnarray}
&& u_y(x,y,t)^2=\int^y_m \partial_s [u^2_y(x,s,t)]\,ds\leq 2
\int^y_0
|u_y(x,y,t)u_{yy}(x,y,t)|\,dy\nonumber\\
&&\leq 2(\int_0^L u_y^2\,dy)^{\frac{1}{2}}(\int_0^L
u_{yy}^2\,dy)^{\frac{1}{2}}.
\end{eqnarray}
Hence,
\begin{equation}
\displaystyle\sup_D u_y^2(x,y,t)\leq \int_0^L u_y^2\,dy+\int_0^L
u_{yy}^2\,dy\equiv \rho^2(x,t).\nonumber
\end{equation}

On the other hand,
\begin{eqnarray}
&&\displaystyle\sup_{x\in R^+} \rho^2(x,t)\leq
\int_0^{\infty}\rho^2(x,t)\,dx+\int_0^{\infty}\rho_x^2(x,t)\,dx\nonumber\\
&&\leq\displaystyle\int_D [u^2_y(x,y,t)+|\nabla
u_y(x,y,t)|^2+u^2_{yyx}(x,y,t)]\,dxdy.
\end{eqnarray}
The proof of Proposition 3.3 is complete.
\end{proof}

{\bf ESTIMATE V.} To estimate $u^N_t$, we differentiate (\ref{eN})
with respect to $t$, multiply the $j$-equation of the resulting
system by $g_{jt}$, sum up over $j=1,...,N$ and integrate over
$R^+$. Calculations, similar to those exploited in Estimate II,
imply

\begin{eqnarray}
&&\frac{d}{dt}(e^{kx},u^2_t)(t)+2ak(e^{kx},u^2_{xt})(t)+5k(e^{kx},u_{txx}^2)(t)\nonumber\\
&&+\int_0^L u_{t xx}^2(0,y,t)\,dy +k(e^{kx},u^2_{yt})(t)\nonumber\\
&&+(k^5-k^3)(e^{kx},u^2_t)(t)=2k(e^{kx},[u u_x]_t,u_t)(t).
\label{5.1}
\end{eqnarray}

We estimate the nonlinear term as follows:
\begin{eqnarray*}
&& I=([uu_x]_t, e^{kx}u_t)(t)=([uu_t]_x, e^{kx}u_t)(t)\\
&&=-k\underbrace{(e^{kx},u_t^2u)(t)}_{I_{1}} -\underbrace{
(e^{kx}uu_t,u_{tx})(t)}_{I_{2}}.
\end{eqnarray*}
 By (\ref{p1}) and (\ref{p2}), for all $\delta
>0$
\begin{eqnarray*}
&&I_{1}=(e^{kx}u_t^2,u)(t) \leq C(\delta){\|
u\|}^2(t){\|e^{\frac{kx}{2}}u_t  \|}^2(t)
+\delta(1+\frac{k^2}{2}){\|e^{\frac{kx}{2}}u_t  \|}^2(t) \\
&&\qquad +2\delta{\|e^{\frac{kx}{2}}u_{tx}  \|}^2(t)+ \delta {\|e^{\frac{kx}{2}}u_{ty}  \|}^2(t), \\
&&I_{2}=(e^{kx}uu_t,u_{tx})(t)= \dfrac{1}{2}(e^{kx}u,{(u_t^2)}_x)(t) \leq C(\delta,k)[{\| u \|}^2(t)\\
&& +{\| u_x \|}^2(t)]{\| e^{\frac{kx}{2}}u_t \|}^2(t)\\
&&  +2\delta{\|e^{\frac{kx}{2}}u_{tx}  \|}^2(t)+ \delta
{\|e^{\frac{kx}{2}}u_{ty}
\|}^2(t)+\delta(1+\frac{k^2}{2}){\|e^{\frac{kx}{2}}u_t\|}^2(t).
\end{eqnarray*}
Substituting $I$ into (\ref{5.1}), taking  $\delta>0$ sufficiently
small and making use of Estimates I-IV and the Gronwall lemma, we
find

\begin{eqnarray*}
(e^{kx},u_t^2)(t)& \leq  (e^{kx},u_{t}^2)(0)e^{C(k,T)(e^{kx},u_0^2)}
 \leq &C(k,T,\|u_0\|)J_w.\nonumber
\end{eqnarray*}
Returning to (\ref{5.1}), we deduce
\begin{eqnarray}
&&(e^{kx},|u^N_t|^2)(t)+\int_0^t\int_0^L
|u^N_{xxs}(0,y,s)|^2\,dyds\nonumber\\
&&+\int_0^t(e^{kx},|u^N_{xs}|^2+|u^N_{xxs}|^2+|u^N_{ys}|^2)(s)\,ds
\leq C(k,T)J_w. \label{Ev}
\end{eqnarray}

{\bf ESTIMATE VI.}

Dropping the index $N$, transform the scalar product
$$(e^{kx}(\partial^4_x u^N-2u_{xx}^N),[u^N_t+u^N u^N_x+\partial^3_x
u^N-\partial^5_x u^N+\partial^2_y u^N_x])(t)=0
$$
into the following equality:

\begin{eqnarray}
&&\frac{d}{dt}(e^{kx},\frac{1}{2}u_{xx}^2+u_x^2)(t)+k(e^{kx},\frac{1}{2}|\partial^4_x
u|^2+u_{xxx}^2+u_{xx}^2)(t)\nonumber\\
&&+\displaystyle\int_0^L\{\frac{1}{2}|\partial_x^4
u(0,y,t)|^2+u_{xxx}^2(0,y,t)+u_{xx}^2(0,y,t)\}\,dy\nonumber\\
&&=-(e^{kx}u_t,2ku_x+k^2u_{xx})(t)-2k(e^{kx}u_{xt},u_{xx})(t)\nonumber\\
&&(+e^{kx}\partial_x^4 u,
2ku_{xx}-u_{xyy})(t)+\frac{k}{2}(e^{kx},u_{xxx}^2)(t)\nonumber\\
&&+\frac{1}{2}\displaystyle\int_0^L
u_{xxx}^2(0,y,t)\,dy+2\displaystyle\int_0^L \partial_x^4
u(0,y,t)u_{xx}(0,y,t)\,dy\nonumber\\
&&+(e^{kx} uu_x, 2u_{xx}-\partial_x^4 u)(t). \label{e6}
\end{eqnarray}

Making use of (\ref{psup}), we estimate the last scalar product in
the right-hand side of (3.20) as follows:
\begin{eqnarray*}
&&(e^{kx} uu_x, 2u_{xx}-\partial_x^4 u)(t)\leq
\delta(e^{kx},|\partial_x^4 u|^2)(t)\\
&&+(\frac{1}{4\delta}+1)(e^{kx},[u^2+|\nabla
u|^2+u_{xy}^2])(t)(e^{kx},u_{xx}^2+u_x^2)(t),
\end{eqnarray*}
where $\delta$ is an arbitrary positive number. Using the Young
inequality ($ab\leq \delta a^2+\frac{1}{4\delta}b^2)$, choosing
$\delta$ sufficiently small and integrating (3.20), we come to the
inequality
\begin{eqnarray}
&&(e^{kx},\frac{1}{2}u_{xx}^2+u_x^2)(t)+\displaystyle\int_0^t\displaystyle\int_0^L
\big[|\partial_x^4 u(0,y,s)|^2+u_{xxx}^2(0,y,s)\big]\,dy ds
\nonumber\\
&&+\displaystyle\int_0^t(e^{kx},u_{xxx}^2+|\partial_x^4
u|^2)(s)\,ds\leq (e^{kx},u_{0x}^2+u_{0xx}^2)\nonumber\\
&&+C(k,T)[\displaystyle\int_0^t(e^{kx},[u^2+|\nabla u|^2
+u_{xy}^2])(s)(e^{kx},u_x^2+u_{xx}^2)(s)\,ds\nonumber\\
&&+\displaystyle\int_0^t\big[(e^{kx},\{|\nabla
u|^2+u_{xx}^2+u_s^2+u_{xs}^2+u_{xyy}^2\})(s)\nonumber\\
&&+\displaystyle\int_0^L u_{xx}^2(0,y,s)\,dy\big]\,ds]. \label{e7}
\end{eqnarray}
Due to previous estimates, $(e^{kx}, u^2+|\nabla
u|^2+u_{xy}^2)(s)\in L^1(0,T).$ Hence, the Gronwall lemma and
(\ref{e7}) imply that
\begin{eqnarray}
&&(e^{kx},u_{xx}^2+u_x^2)(t)+\displaystyle\int_0^t(e^{kx},u_{xxx}^2+|\partial_x^4
u|^2)(s)\,ds\nonumber\\
&&+\displaystyle\int_0^t \displaystyle\int_0^L \big[|\partial_x^4
u(0,y,s)|^2+u_{xxx}^2(0,y,s)\big]\,dy ds\nonumber\\&& \leq C(k,T)J_w
.\label{e8}
\end{eqnarray}
Now from the equality

$$-(e^{kx}\partial^5_x u^N,[u^N_t+u^N u^N_x+\partial^3_x
u^N-\partial^5_x u^N+\partial^2_y u^N_x])(t)=0
$$
we find that

$$\displaystyle\int_0^t(e^{kx},|\partial_x^5 u|^2)(s)\,ds\leq
C(k,T)J_w.
$$
Taking into account (\ref{e8}), we obtain

\begin{eqnarray}
&&e^{kx},|u^N_{xx}|^2+|u^N_x|^2)(t)+\displaystyle\int_0^t(e^{kx},|u^N_{xxx}|^2+|\partial_x^4
u^N|^2+|\partial_x^5 u^N|^2)(s)\,ds\nonumber\\
&&+\displaystyle\int_0^t \displaystyle\int_0^L \big[|\partial_x^4
u^N(0,y,s)|^2+|u^N_{xxx}|^2(0,y,s)\big]\,dy ds \nonumber\\&&\leq
C(k,T)J_w \label{EVI}
\end{eqnarray}
with the constant independent of $N.$

{\bf ESTIMATE VII.} Omitting the index $N$, we deduce from the
scalar product
$$-2(e^{kx}u_{yy}^N,[u^N_t+u^N u^N_x+\partial^3_x
u^N-\partial^5_x u^N+\partial^2_y u^N_x])(t)=0
$$
the following equality:
\begin{eqnarray}
&&
5k(e^{kx},u_{xxy}^2)(t)+2ak(e^{kx},u_{xy}^2)(t)+\displaystyle\int_0^L
u_{xxy}^2(0,y,t)\,dy\nonumber\\
&&+(e^{kx},u_{yy}^2)(t)+(k^5-k^3)(e^{kx},u^2_y)(t)
\nonumber\\&&=2(e^{kx}[u_t+uu_x],u_{yy})(t). \label{e9}
\end{eqnarray}

The term $I=2(e^{kx}uu_x,u_{yy})(t)$ may be estimated as
\begin{eqnarray*}
&&I\leq \frac{1}{\delta}(e^{kx},u_{yy}^2)(t)+\delta \sup _D
u^2(x,y,t)(e^{kx},u_x^2)(t)\\
&&\leq\delta (e^{kx},u_{xy}^2)(t)(e^{kx},u_x^2)(t)+\delta
(e^{kx},u^2+|\nabla
u|^2)(t)(e^{kx},u_x^2)(t)\nonumber\\&&+\frac{1}{\delta}(e^{kx},u_{yy}^2)(t).
\end{eqnarray*}
Taking into account (\ref{EVI}) and choosing $\delta>0$ sufficiently
small, we find
\begin{equation}
(e^{kx},u_{xy}^2+u_{xxy}^2)(t)+\displaystyle\int_0^L
u_{xxy}^2(0,y,t)\,dy\leq C(k,T,J_w)J_w. \label{EVII}
\end{equation}

Jointly, Estimates I-VI read

\begin{eqnarray}
&&(e^{kx}, \big[|u^N|^2+|u^N_t|^2+|\nabla u^N|^2+|\nabla
u^N_x|^2+|\nabla u^N_y|^2+|u^N_{xxy}|^2\big ])(t)\nonumber\\
&&+\displaystyle\int_0^L|u^N_{xxy}(0,y,t)|^2\,dy+\displaystyle\int_0^t(e^{kx},
\big[|\nabla u^N_s|^2+|\nabla
u^N|^2+|\nabla u^N_x|^2\nonumber\\
&&+|\nabla u^N_y|^2+|\nabla u^N_{xx}|^2+|\nabla
u^N_{yy}|^2+|\partial_x^4 u^N|^2+|\partial_x^5 u^N|^2+|\partial_x^2
u^N_{yy}|^2 \big ])(s)\,ds\nonumber\\
&&+\displaystyle\int_0^t\displaystyle\int_0^L[|u^N_{xxy}(0,y,s)|^2+|u^N_{xxyy}(0,y,s)|^2\nonumber\\
&&+|\partial_x^4 u^N(0,y,s)|^2+|\partial_x^3 u^N(0,y,s)|^2
]\,dyds\nonumber\\&&\leq C(k,T,J_w)J_w, \label{EF}
\end{eqnarray}
where the constant $C(k,T,J_w)$ does not depend on $N$.

\section { Passage to the limit as $N$ tends to $\infty$.}

Uniform in $N$ estimate (\ref{EF}) and standard arguments imply that
there exists a function $u(x,y,t)=\lim_{N\to \infty}u^N(x,y,t)$ \;
such that

\begin{eqnarray*}
&& u\in L^{\infty}(0,T;H^2(D))\cap L^2(0,T;H^3(D));\,
\partial_x^4 u,\;\partial_x^5 u \in L^2(0,T;L^2(D)),\\
&&u_{xxy}\in L^{\infty}(0,T;L^2(D)),\,\; u_t\in
L^{\infty}(0,T;L^2(D))\cap L^2(0,T;H^1(D))
\end{eqnarray*}
and $u(x,y,t)$ satisfies the following integral identity:
\begin{equation}
\displaystyle\int_0^T\int_D\left[u_t+uu_x+\Delta u_x-\partial_x^5
u\right]\psi(x,y,t)\,dxdydt=0, \label{solut}
\end{equation}
where $\psi(x,y,t)$ is an arbitrary function from $L^2(D_T)$.
Obviously, $u(x,y,t)$ is a solution to the problem (2.1)-(2.4) and
satisfies estimate (3.26). It follows from (\ref{solut}) and (3.26)
that
$$ \partial_x^5 u+\Delta u_x \in L^{\infty}(0,T;L^2(D))\cap
L^2(0,T;H^1(D)).$$ This proves the existence part of Theorem
\ref{T1}.

\section{Uniqueness}
 Let $u_1$ and $u_2$ be distinct solutions of
$(\ref{1.1})-(\ref{1.3})$ and $z=u_1-u_2$. Then $z(x,y,t)$ satisfies
the following initial boundary value problem:

\begin{eqnarray}
&& Lz= z_t +  \dfrac{1}{2}(u_1^2-u_2^2)_x + \partial_x^3z -\partial_x^5 z +\partial_y^2 z_x=0 \ \textrm{in} \ Q_t; \label{u1} \\
&& z(0,y,t)= z_x(0,y,t)=z(x,0,t)=z(x,L,t)=0, \nonumber \\
& & \qquad \qquad \qquad y \in (0,L), \quad x>0, \quad t>0; \label{u2}\\
&& z(x,y,0)=0, \quad (x,y) \in D. \label{u3}
\end{eqnarray}

From the scalar product
\begin{equation}
2(Lz, e^{kx}z)(t)=0, \label{eu}
\end{equation}
 acting in the same manner as by the proof of Estimate II and using Proposition 3.3, we obtain
\begin{eqnarray*}
&&I_1= 2\displaystyle\int_De^{kx}z_tz\,dxdy=
\dfrac{d}{dt}\displaystyle\int_De^{kx}z^2\,dxdy,\\
&&I_2=2\displaystyle\int_De^{kx}z\left[\partial_x^3z+\partial^2_y
z_x
-\partial_x^5z \right]\,dxdy\nonumber\\
&&=\displaystyle\int_{0}^{L}z_{xx}^2(0,y,t)\,dy
+k\displaystyle\int_D\left[2az_x^2+z_y^2+5z^2_{xx}\right]\,dxdy,\\
&&I_3=(\left[u_1^2-u_2^2\right]_x,e^{kx}z)(t)\leq
\delta(e^{kx},z_x^2)(t)\\
&&+C(\delta,k)\sum_{i=1}^2\left[\|u_i\|^2(t)+\|\nabla
u_i\|^2(t)+\|u_{ixy}\|^2(t)\right](e^{kx},z^2)(t).
\end{eqnarray*}
Substituting $I_1-I_3$ into (\ref{eu}) and taking $\delta>0$
sufficiently small, we come to the inequality
$$\dfrac{d}{dt}(e^{kx},z^2)(t) \leq C(k)\sum_{i=1}^2\left[\|u_i\|^2(t)+\|\nabla
u_i\|^2(t)+\|u_{ixy}\|^2(t)\right](e^{kx},z^2)(t).$$ Taking into
account that by (3.26) $(e^{kx},|u_i|^2+|\nabla
u_i|^2+|u_{ixy}|^2)(t)  \in L^1(0,T)$ $(i=1,2)$ and (\ref{u3}), we
get $\|z\|(t)\equiv 0$ for $a.e.\:t \in (0,T).$ This proves
uniqueness of a regular solution of (\ref{1.1})-(\ref{1.3}) and
completes the proof of Theorem 3.1.

\end{proof}

\section{Decay of Solutions}

In order to study the behavior of solutions while $t\to \infty,$ it
is necessary to consider the presence of the linear transport term
$u_x$, because this term is crucial for the appearance of critical
sets where  decay of solutions may fail to exist \cite{rosier2}.

\begin{teo} \label{tdec}
Let $\alpha =1$  and $L,k$ be real positive numbers such that $L \in
(0,\pi)$,\;  $k^2 <\min(\frac{3}{5}, \frac{4\delta^2}{9})$. Given
$u_0(x,y)$ such that
$$u_0(0,y,t)=u_{0x}(0,y,t)=u_0(x,0,t)=u_0(x,L,t)=0$$ and
$$\|u_0\|^2 \leq \dfrac{9\delta^2}{2}\min(\frac{1}{8},
\frac{a}{2},\frac{\delta^2}{4}),$$
 where
$$\delta^2=\frac{\pi^2-L^2}{4L^2} .$$  Then regular solutions of (\ref{1.1})-(\ref{1.3}) satisfy the
inequality
\begin{eqnarray*}
&&{\|u\|}_{H^1(D)}^2(t)+{\|\partial_x^2 u\|}^2(t)\\
&&\leq C(k,\chi,(e^{kx},u_0^2))(1+t)e^{-\chi t} (e^{kx},
u_0^2+|\nabla u_0|^2+u_{0xx}^2+|u_0|^3),
\end{eqnarray*}
where\qquad $\chi=k(\delta^2+k^4).$
\end{teo}

\begin{proof}
\begin{lem}\label{ld1}
Let all the conditions of Theorem 6.1 be fulfilled. Then regular
solutions of (\ref{1.1})-(\ref{1.3}) satisfy the inequality
$${\|u\|}^2(t)\leq(e^{kx},u^2)(t) \leq e^{-\chi t} (e^{kx}, u_0^2),$$
where $\chi=k(\delta^2+k^4).$
\end{lem}
\begin{proof}
Transform the integral
\begin{eqnarray}
&&(u,Lu)(t)= (u,u_t)(t) + (u,u_x)(t)+(u^2,u_x)(t)\nonumber\\
&&+(u,\Delta u_x)(t)-(\partial_x^5 u,u)(t)=0 \label{e6}
\end{eqnarray}
into the equality
$${\|u\|}^2(t)+ \displaystyle\int_{0}^{t}\displaystyle\int_{0}^{L} u_{xx}^2(0,y,\tau)\,dyd\tau= {\|u_0\|}^2,$$
whence
\begin{equation}
{\|u\|}^2(t) \leq {\|u_0\|}^2, \quad t > 0. \label{2.2}
\end{equation}
Next, consider for $k$ defined in conditions of Theorem 6.1 the
equality
 \begin{eqnarray}
 &&(e^{kx}u,
Lu)(t)=(e^{kx}u, u_t)(t)+ (e^{kx}u, u_x)(t) \nonumber
\\
&&+(e^{kx}u^2, u_x)(t)+ (e^{kx}u,\Delta u_x-\partial_x^5 u)(t)=0
\nonumber
\end{eqnarray} which can be reduced to the form
\begin{eqnarray}
&& \dfrac{d}{dt}(e^{kx}, u^2)(t) + k(e^{kx}, u_y^2+2a u_x^2
+5u_{xx}^2)(t)
+ \displaystyle\int_{0}^{L} u_{xx}^2(0,y,t)\,dy\nonumber \\
&& -(k+k^3-k^5)(e^{kx}, u^2)(t)-\dfrac{2k}{3}(e^{kx}, u^3)(t)=0.
\label{e61}
\end{eqnarray}
Using (\ref{p1}),  we calculate
\begin{eqnarray}
&&I=- \dfrac{2k}{3}(e^{kx}, u^3)(t)  \leq  \dfrac{2k}{3}{\|u\|}(t){\|e^{kx/2}u\|}_{L^4(D)}^2(t) \nonumber \\
&&  \leq \dfrac{4k}{3}{\|u\|}(t){\|e^{kx/2}u\|}(t){\|\nabla
(e^{kx/2}u)\|}(t). \nonumber
\end{eqnarray}
Taking into account (\ref{2.2}),
$$I \leq \dfrac{4k}{3}{\|u_0\|}{\|e^{kx/2}u\|}(t) \left\{ (e^{kx}, [u_y^2+ \frac{k^2}{2} u^2 + 2u_x^2])(t) \right\}^{1/2}.$$
By the Young inequality,
\begin{equation}
I  \leq \epsilon k(e^{kx}, 2u_y^2 + 4 u_x^2+  k^2 u^2)(t)+
\dfrac{2k}{9\epsilon}{\|u_0\|}^2(e^{kx}, u^2)(t), \label{2.6}
\end{equation}
where $\epsilon$ is an arbitrary positive number. \\
Taking $0< \epsilon <\min(\frac{1}{8},\frac{a}{2}),$ we reduce
(\ref{e61}) to the following inequality:
\begin{eqnarray}
 & & \dfrac{d}{dt} (e^{kx}, u^2)(t) - (k+\epsilon k^3
+k^3-k^5)(e^{kx}, u^2)(t) + k(2a-4\epsilon)(e^{kx},u_x^2)(t) \nonumber \\
&&  + k(1-2\epsilon) (e^{kx}, u_y^2)(t) -
\dfrac{2k}{9\epsilon}{\|u_0\|}^2 (e^{kx}, u^2)(t) \leq 0.
\label{2.7}
\end{eqnarray}
The following proposition is crucial for our proof.
\begin{prop}\label{propcr}
Let $L>0$ be a finite  number and  $u(x,y,t)$ be a regular solution
to (2.1)-(2.4). Then
\begin{equation}\label{7.1}
\int_{R^+}\int_0^L e^{kx}u^2(x,y,t)\,dy\,dx\leq
\frac{L^2}{\pi^2}\int_{R^+}\int_0^L e^{kx}u_y^2(x,y,t)\,dy\,dx.
\end{equation}
\end{prop}
\begin{proof}
Since $u(x,0,t)=u(x,L,t)=0,$ fixing $x,t$, we can use with respect
to $y$ the following Steklov inequality: if $f(y)\in H^1_0(0,\pi)$
then
$$\int_0^{\pi}f^2(y)\,dy\leq \int_0^{\pi} |f_y(y)|^2\,dy.$$
After a corresponding process of scaling we prove  Proposition
\ref{propcr}.
\end{proof}

Making use of \eqref{7.1}, we get

\begin{eqnarray}
&& \dfrac{d}{dt} (e^{kx}, u^2)(t) +k
\left[\dfrac{\pi^2}{L^2}-1-\dfrac{2\pi^2\epsilon}{L^2}
-\epsilon k^2-k^2+k^4\right](e^{kx}, u^2)(t) \nonumber \\
&&+2(a-2\epsilon)(e^{kx},u_x^2)(t)  -
\dfrac{2k}{9\epsilon}{\|u_0\|}^2 (e^{kx}, u^2)(t) \leq 0.\nonumber
\end{eqnarray}
Denoting
\begin{equation} \dfrac{\pi^2}{L^2}-1=4 {\delta}^2 >0 \label{2.9} \end{equation}
and taking
$$\epsilon<\min(\frac{1}{8},\frac{a}{2},\frac{\delta^2}{4}),\qquad k^2<\min(\frac{3}{5},\frac{4\delta^2}{9}),$$
we find
$$\dfrac{d}{dt} (e^{kx},
u^2)(t)+2k(\delta^2-\frac{\|u_0\|^2}{9\epsilon}+\frac{k^4}{2})(e^{kx},u^2)(t).$$
By the conditions of Lemma 6.2,
$$ \frac{\|u_0\|^2}{9\epsilon}\leq\frac{\delta^2}{2},$$
hence
$$\frac{d}{dt}(e^{kx},u^2)(t)+\chi(e^{kx},u^2)(t)\leq 0,$$
where $\chi=k(\delta^2+k^4).$\\
This implies
$${\|u\|}^2(t)\leq(e^{kx},u^2)(t) \leq e^{-\chi t} (e^{kx}, u_0^2),$$
The proof of Lemma \ref{ld1} is complete. \end{proof}

\begin{prop} \label{pd}Regular solutions of (\ref{1.1})-(\ref{1.3}) satisfy the
inequality
\begin{eqnarray}
&&\|u_{xx}\|^2(t)+\|\nabla u\|^2(t) -\frac{1}{3}\displaystyle\int_D
u^3\,dxdy\nonumber\\
&&  \leq\|u_{0xx}\|+\|\nabla u_0\|^2-\frac{1}{3}\displaystyle\int_D
u_0^3\,dxdy. \label{6.1}
\end{eqnarray}
\end{prop}
\begin{proof} Estimate separate terms in the following scalar
product:
\begin{eqnarray}
&& -2(u_t+u_x +uu_x+u_{xxx}+u_{xyy}-\partial_x^5 u,\nonumber\\
&&[u_{xx} +u_{yy}-\partial_x^4 u+\frac{u^2}{2}])(t)=0. \label{6.2}
\end{eqnarray}

That is
\begin{eqnarray}
&& I_1=-2(u_t,u_{xx} +u_{yy}-\partial_x^4 u+\frac{u^2}{2})(t)\nonumber\\
&&=\frac{d}{dt}[\|u_{xx}\|^2(t)+\|\nabla u\|^2(t)
-\frac{1}{3}\displaystyle\int_D u^3\,dxdy] ,
\nonumber\\
&& I_2=-2(u_x,u_{xx} +u_{yy}-\partial_x^4
u+\frac{u^2}{2})(t)=\displaystyle\int_0^L
u_{xx}^2(0,y,t)\,dy.\nonumber\\
&& I_3=-2(uu_x,u_{xx} +u_{yy}-\partial_x^4 u+\frac{u^2}{2})(t) =(u^2,u_{xxx}+u_{xyy}-\partial_x^5 u)(t),\nonumber\\
&& I_4=-(u^2,u_{xxx}+u_{xyy}-\partial_x^5 u)(t)=-I_3,\nonumber\\
&& I_5=-2(-\partial_x^5 u,u_{xx}+u_{yy}-\partial_x^4 u)(t)
=\displaystyle\int_0^L\{|\partial_x^4
u(0,y,t)|^2\nonumber\\
&&-2\partial_x^4u(0,y,t)u_{xx}(0,y,t)+u^2_{xxy}(0,y,t)+u^2_{xxx}(0,y,t)\}\,dy,\nonumber\\
&& I_6=-2(u_{xxx},u_{xx}+u_{yy}-\partial_x^4
u)(t)\nonumber\\
&&=-\displaystyle\int_0^L
[u_{xxx}^2(0,y,t)-u_{xx}^2(0,y,t)]\,dy.\nonumber
\end{eqnarray}
It is easy to see that
$$I_3+I_4=0,\qquad I_2+I_5+I_6\geq 0.$$
Hence
$$ \frac{d}{dt}[\|u_{xx}\|^2(t)+\|\nabla u\|^2(t)
-\frac{1}{3}\displaystyle\int_D u^3\,dxdy]\leq 0$$ which implies
(\ref{6.1}). The proof of Proposition \ref{pd} is complete.
\end{proof}

\begin{lem} \label{ld2} Let all the conditions of Theorem \ref{tdec} be fulfilled.
Then  regular solutions of (\ref{1.1})-(\ref{1.3}) satisfy the
inequality
\begin{eqnarray*}
&&\|u_{xx}\|^2(t)+\|\nabla u\|^2(t)\leq C(1+t)e^{-\chi t}(e^{kx},
\\
&&[u_0^2+|\nabla u_0|^2+ u^2_{0xx}+|u_0|^3]).
\end{eqnarray*}
\end{lem}
\begin{proof}

Acting in the same manner as by the proof of Proposition \ref{pd},
we get from the scalar product
$$-2(e^{\chi t}Lu,u_{xx} +u_{yy}-\partial_x^4
u+\frac{u^2}{2})(t)=0$$
the following inequality:
\begin{eqnarray}
&&e^{\chi t}[\|u_{xx}\|^2(t)+\|\nabla u\|^2(t)
-\frac{1}{3}\displaystyle\int_D u^3\,dxdy]\nonumber\\
&&-\chi\{\displaystyle\int_0^t e^{\chi s}[\|u_{xx}\|^2(s)+\|\nabla
u\|^2(s) -\frac{1}{3}\displaystyle\int_D u^3(x,y,s)\,dxdy]\,ds\}\nonumber\\
&&\leq|u_{0xx}\|+|\nabla u_0\|^2-\frac{1}{3}\displaystyle\int_D
u_0^3\,dxdy. \label{6.3}
\end{eqnarray}
Making use of (\ref{p1}), we get
$$\frac{e^{\chi t}}{3}[\displaystyle\int_D u^3(x,y,t)\,dxdy\leq
\frac{3}{2}\|\nabla u\|^2(t)+\frac{2}{3}\|u\|^4(t)].$$

Substituting this into (\ref{6.3}) reads

\begin{eqnarray}
&&e^{\chi t}[\|u_{xx}\|^2(t)+\|\nabla u\|^2(t)]\leq  \frac{4e^{\chi
t}}{9}\|u\|^4(t)\nonumber\\
&&+\frac{4\chi}{3}\displaystyle\int_0^t e^{\chi
s}[2\|u_{xx}\|^2(s)+2\|\nabla u\|^2(s)+\|u\|^4(s)]\,ds\nonumber\\
&&+2[\|u_{0xx}\|^2+2\|\nabla
u_0\|^2+\frac{1}{3}\|u_0\|^3_{L^3(D)}].\label{6.4}
\end{eqnarray}

From the scalar product
$$2(e^{\chi t}Lu,e^{kx}u)(t)=0$$
we deduce
\begin{eqnarray}
&&e^{\chi t}(e^{kx},u^2)(t)-\displaystyle\int_{0}^{t}e^{\chi
\tau}(\chi+k+k^3-k^5)(e^{kx},u^2)(\tau)\,d\tau\nonumber\\
&&+k\displaystyle\int_{0}^{t}e^{\chi
\tau}[2a(e^{kx},u_x^2)(\tau)+(e^{kx},u_y^2)(\tau)+5(e^{kx},u^2_{xx})(\tau)]\,d\tau\nonumber\\
&&-\frac{2k}{3}\displaystyle\int_{0}^{t}e^{\chi \tau}(e^{kx},
 u^3)(\tau)\,d\tau
+\displaystyle\int_{0}^{t}e^{\chi\tau}\displaystyle\int_{0}^{L}u_{xx}^2(0,y,\tau)dyd\tau\nonumber\\&&=(e^{kx},u^2_0).
\label{6.6}
\end{eqnarray}
Making use of (\ref{p1}), we estimate
\begin{eqnarray}
&&I=-\frac{2k}{3}(e^{kx},u^3)(t) \leq
\frac{4k}{3}(e^{kx},u^2)(t)\|\nabla(e^{\frac{kx}{2}}u)\|(t)\nonumber\\
&&\leq 2k\delta(e^{kx},|\nabla
u|^2+u^2_{xx})(t)+\frac{k^3}{2}\delta(e^{kx},u^2)(t)+\frac{24k}{9\delta}(e^{kx},u^2)^2(t).\nonumber
\end{eqnarray}

Taking $\delta=\frac{1}{4}\min(1,2a)$ and using Lemma 6.2, we obtain
\begin{equation}
I\leq \frac{k}{2}(e^{kx},|\nabla
u|^2+u^2_{xx})(t)+C(k)(e^{kx},u^2)(t).\nonumber
\end{equation}

Substituting $I$ into (\ref{6.6}) gives
\begin{eqnarray}
&&e^{\chi
t}(e^{kx},u^2)(t)+\frac{k}{2}\displaystyle\int_{0}^{t}e^{\chi
\tau}(e^{kx},|\nabla u|^2+u^2_{xx})(\tau)\,d\tau\nonumber\\
&&\leq C(k,\chi)\displaystyle\int_{0}^{t}e^{\chi
\tau}(e^{kx},u^2)(\tau)\,d\tau+(e^{kx},u_0^2). \label{6.10}
\end{eqnarray}
By Lemma 6.2,
$$(e^{kx},u^2)(t)\leq e^{-\chi t}(e^{kx},u_0^2)$$
which implies
$$\displaystyle\int_{0}^{t}e^{\chi
\tau}(e^{kx},u^2)(\tau)\,d\tau\leq t(e^{kx},u_0^2).$$
 Returning to
(\ref{6.10}), we get
 \begin{eqnarray}
&& e^{\chi t}[\|\nabla u\|^2(t)+\|u_{xx}\|^2(t)]
+\frac{8\chi}{3}\displaystyle\int_{0}^{t}e^{\chi
 \tau}(e^{kx},|\nabla u|^2+u^2_{xx})(\tau)\,d\tau\nonumber\\
&&\leq\frac{4}{9}e^{\chi t}\|u\|^4(t)
+\frac{4\chi}{3}\displaystyle\int_{0}^{t}e^{\chi
 \tau}(e^{kx},u^2)^2(\tau)\,d\tau\nonumber\\&&+2(e^{kx},|\nabla u_0|^2+|u_0|^3+u^2_{0xx}).\label{6.11}
 \end{eqnarray}
 Again by Lemma 6.2,
 $$e^{\chi t}\|u\|^4(t)\leq
 e^{-\chi t}(e^{kx},u_0^2)^2\leq (e^{kx},u_0^2)^2,$$
$$\displaystyle\int_{0}^{t}e^{\chi
 \tau}(e^{kx},u^2)^2(\tau)\,d\tau \leq C(\chi)(e^{kx},u_0^2),$$
and from (\ref{6.11})
$$\displaystyle\int_{0}^{t}e^{\chi \tau}(e^{kx},|\nabla
u|^2+u^2_{xx})(\tau)\,d\tau\leq C(1+t)(e^{kx},u^2_0+ |\nabla
u_0|^2+|u_0|^3+u^2_{0xx}).
$$

Then (\ref{6.11}) becomes
$$\|\nabla u\|^2(t)+\|u_{xx}\|^2(t)\leq C(1+t) e^{-\chi
t}(e^{kx},u_0^2+|\nabla u_0|^2+|u_0|^3+u^2_{0xx}).$$ This proves
Lemma \ref{ld2}.
\end{proof} Hereby, the proof of Theorem \ref{tdec} is complete.
\end{proof}
\vspace{0.2cm}
\par In the case $\alpha=0$ we have the following
result:

\vspace{0.2cm}

\begin{teo}
Let $\alpha=0$, $L>0$, $k^2<\min(\frac{3}{5},\frac{\pi^2}{5L^2})$.
Given $u_0(x,y)$ such that
$$u_0(0,y,t)=u_{0x}(0,y,t)=u_0(x,0,t)=u_0(x,L,t)=0$$ and
$\|u_0\|^2\leq
\dfrac{9\pi^2}{16L^2}\min(\dfrac{1}{4},\dfrac{a}{2})$. Then regular
solutions of (\ref{1.1})-(\ref{1.3}) satisfy the following
inequality
$$ \|u\|^2_{H^1(D)}(t)+\|u_{xx}\|^2(t) \leq C(1+t) e^{- \chi t}(e^{kx}, u_0^2+|\nabla u_{0}|^2+u_{0xx}^2+|u_0|^3),$$
where $\chi=k(\dfrac{\pi^2} {8L^2}+k^4)$.
\end{teo}
\begin{proof} Repeating the proof of Lemma 6.2, we come to the
following inequality:
\begin{eqnarray*}
&&\dfrac{d}{dt}(e^{kx},u^2)(t)+k[\frac{\pi^2}{L^2}(1-2\epsilon)-k^2(1+\epsilon)+k^4](e^{kx},u^2)(t)\\
&&+2(a-2\epsilon)(e^{kx},u^2_x)(t)-\frac{2k}{9\epsilon}\|u_0\|^2(e^{kx},u^2)(t)\leq
0.
\end{eqnarray*}
Taking \quad $\epsilon < \min(\frac{1}{4},\;\frac{a}{2}),$ we get
\begin{equation}
\dfrac{d}{dt}(e^{kx},u^2)(t)+k[\frac{\pi^2}{2L^2}-k^2(1+\epsilon)+k^4](e^{kx},u^2)(t)\nonumber\\
-\frac{2k}{9\epsilon}\|u_0\|^2(e^{kx},u^2)(t)\leq 0.
\end{equation}
Since\quad $ \epsilon< \frac{1}{4},$  putting
$$ k^2< \min(\frac{3}{5},\;\frac{\pi^2}{5L^2}), $$
we obtain

\begin{equation}
\dfrac{d}{dt}(e^{kx},u^2)(t)+k[\frac{\pi^2}{4L^2}+k^4-\frac{2}{9\epsilon}\|u_0\|^2](e^{kx},u^2)(t)\leq
0.\nonumber
\end{equation}

By the conditions of Theorem 6.6, \quad $\|u_0\|^2<
\dfrac{9\pi^2\epsilon}{16L^2}. $ This implies
$$\dfrac{d}{dt}(e^{kx},u^2)(t)+\chi(e^{kx},u^2)(t)\leq 0,$$
where $$\chi=k(\dfrac{\pi^2}{8L^2}+k^4).$$ Hence,
$$\|u\|^2(t)\leq (e^{kx},u^2)(t)\leq e^{-\chi t}(e^{kx},u_0^2).$$
The rest of the proof of Theorem 6.6 is a simple repetition of the
proof of Theorem 6.1.
 \end{proof}

\begin{rem} The presence in (\ref{1.1}) of the linear term
$u_x$ \;(the case $\alpha =1$) implies a restriction for value of
$L:(L < \pi)$ which means that a channel $D$ has limitations on its
width. On the other hand,
 absence of this term (the case  $\alpha=0$) allows  $L$ to be any finite positive number; it means that a channel
 may be of any finite width.
\end{rem}

\end{document}